\documentclass[a4paper,10pt]{amsart}
\usepackage[top=0.8in, bottom=0.8in, left=1.25in, right=1.25in]{geometry}
\usepackage{amsmath}
\usepackage{amsfonts}
\usepackage{amssymb}
\usepackage{graphicx}
\usepackage{mathrsfs}
\usepackage{pb-diagram}
\usepackage{epstopdf}
\usepackage{CJK}
\usepackage{slashed}

\usepackage{amscd,amsthm,curves,enumerate,latexsym,multibox}
\usepackage{xypic}
\usepackage[all]{xy}
\usepackage{epstopdf}

\newtheorem{theorem}{Theorem}[section]
\newtheorem{lemma}[theorem]{Lemma}
\newtheorem{proposition}[theorem]{Proposition}
\newtheorem{corollary}[theorem]{Corollary}
\theoremstyle{definition}
\newtheorem{definition}[theorem]{Definition}
\newtheorem{remark}[theorem]{Remark}
\newtheorem{example}[theorem]{Example}

\begin{document}
\title{ Orientability for gauge theories on Calabi-Yau manifolds  }
\author{Yalong Cao}
\address{The Institute of Mathematical Sciences and Department of Mathematics, The Chinese University of Hong Kong, Shatin, Hong Kong}
\email{ylcao@math.cuhk.edu.hk}

\author{Naichung Conan Leung}
\address{The Institute of Mathematical Sciences and Department of Mathematics, The Chinese University of Hong Kong, Shatin, Hong Kong}
\email{leung@math.cuhk.edu.hk}

%\author{Yalong Cao and Naichung Conan Leung}
%\date{  }
%\address{The Institute of Mathematical Sciences and Department of Mathematics, The Chinese University of Hong Kong, Shatin, Hong Kong}
%\email{ylcao@math.cuhk.edu.hk}

%\author{Naichung Conan Leung}
%\address{The Institute of Mathematical Sciences and Department of Mathematics, The Chinese University of Hong Kong, Shatin, Hong Kong}
%\email{leung@math.cuhk.edu.hk}

\maketitle

\begin{abstract}
We study orientability issues of moduli spaces from gauge theories on Calabi-Yau manifolds. Our results generalize and strengthen those for Donaldson-Thomas theory on Calabi-Yau manifolds of dimensions 3 and 4.
We also prove a corresponding result in the relative situation
which is relevant to the gluing formula in DT theory.
\end{abstract}
${}$ 
%\tableofcontents

\section{Introduction}
Donaldson invariants count anti-self-dual connections on closed oriented 4-manifolds \cite{d}. The definition requires an orientablity result
proved by Donaldson in \cite{d1}. Indeed, Donaldson theory fits into the 3-dimensional TQFT structure in the sense of Atiyah \cite{atiyah1}.
In particular, relative Donaldson invariants for $(X,Y=\partial X)$ take values in the
instanton Chern-Simons-Floer (co)homology $HF^{*}_{CS}(Y)$ \cite{donaldson, witten}.
The Euler characteristic of $HF^{*}_{CS}(Y)$ is the Casson invariant which counts flat connections on a closed 3-manifold $Y$.
%When we degenerate $X_{\mathbb{R}}^{4}$ into the union of two 4-manifolds with boundaries gluing along a closed 3-manifold, i.e. $X_{\mathbb{R}}^{+}\cup_{Y_{\mathbb{R}}^{3}}X_{\mathbb{R}}^{-}$, relative Donaldson invariants for $X_{\mathbb{R}}^{\pm}$ define elements in the instanton Chern-Simons-Floer (co)homology $HF^{*}_{CS}(Y_{\mathbb{R}}^{3})$ of $Y_{\mathbb{R}}^{3}$ and their pairing in $HF^{*}_{CS}(Y_{\mathbb{R}}^{3})$ gives the Donaldson invariant of $X_{\mathbb{R}}^{4}$ \cite{donaldson}.

As was proposed by Donaldson and Thomas \cite{dt}, we are interested in the \textit{complexification} of the above theory.
Namely, we consider holomorphic vector bundles (or general coherent sheaves) over Calabi-Yau manifolds \cite{yau}.
The complex analogs of (i) Donaldson invariants, (ii) Chern-Simons-Floer (co)homology $HF^{*}_{CS}(Y)$, and (iii) Casson invariants
%\begin{equation}(i) Donaldson invariants, (ii) Chern-Simons-Floer (co)homology HF^{*}_{CS}(Y), and (iii) Casson invariants \end{equation}
are (i) $DT_{4}$ invariants, (ii) $DT_{3}$ (co)homology $H^{*}_{DT_{3}}(Y)$, and (iii) $DT_{3}$ invariants.

As a complexification of Casson invariants, Thomas defined Donaldson-Thomas invariants for Calabi-Yau 3-folds \cite{th}.
$DT_{3}$ invariants for ideal sheaves of curves are related to many other interesting subjects including Gopakumar-Vafa conjecture on
BPS numbers in string theory \cite{gv, hosono, kl} and MNOP conjecture \cite{mnop, mnop2, moop, pandpixton}
which relates $DT_{3}$ invariants and Gromov-Witten invariants. The generalization of $DT_{3}$ invariants to
count strictly semi-stable sheaves is due to Joyce and Song \cite{js} using Behrend's result \cite{behrend}.
Kontsevich and Soibelman proposed generalized as well as motivic $DT$ theory for Calabi-Yau 3-categories \cite{ks},
which was later studied by Behrend, Bryan and Szendr\"{o}i \cite{bbs} for Hilbert schemes of points.
The wall-crossing formula \cite{ks, js} is an important structure for Bridgeland's stability condition \cite{bridgeland} and
Pandharipande-Thomas invariants \cite{pt, toda}.

As a complexification of Chern-Simons-Floer theory, Brav, Bussi, Dupont, Joyce and Szendroi \cite{bbdjs}, Kiem and Li \cite{kl}
recently defined a cohomology theory on Calabi-Yau 3-folds whose Euler characteristic is the $DT_{3}$ invariant.
The point is that moduli spaces of simple sheaves on Calabi-Yau 3-folds are locally critical points of holomorphic functions \cite{bbj, js},
and we could consider perverse sheaves of vanishing cycles of these functions. They glued these local perverse sheaves and
defined its hypercohomology as $DT_{3}$ cohomology. In general, gluing these perverse sheaves requires
a square root of the determinant line bundle of the moduli space. Nekrasov and Okounkov proved its existence in \cite{no}.
The square root is called an orientation data if it is furthermore compatible with wall-crossing (or Hall algebra structure) \cite{ks}
whose existence was proved by Hua on simply-connected torsion-free $CY_{3}$ \cite{hua1}.

As a complexification of Donaldson theory, Borisov and Joyce \cite{bj} and the authors \cite{cao, caoleung} developed $DT_{4}$
invariants (or 'holomorphic Donaldson invariants') which count stable sheaves on Calabi-Yau 4-folds. To define the invariants,
we need an orientablity result, which was solved by the authors in \cite{caoleung} for Calabi-Yau 4-fold $X$ which satisfies $H^{odd}(X,\mathbb{Z})=0$
(for instance, complete intersections in smooth toric varieties satisfy this condition).
%Later, we generalized this result to torsion-free (i.e. homologies are torison-free) Calabi-Yau 4-folds \cite{caoleung2}.

In this paper, we show that all these orientability results have their origin in spin geometry \cite{adams, lawson}, and then
generalize and strengthen them to Calabi-Yau manifolds of any dimension.  \\

Let us start with a compact spin manifold $X$ of even dimension and a (Hermitian) complex vector bundle $(E,h)\rightarrow X$.
Given an unitary connection $A$ on $E$, one can define the twisted Dirac operator
\begin{equation}\slashed{D}_{A^{*}\otimes A}: \Gamma(\slashed{S}^{+}_{\mathbb{C}}(X)\otimes EndE)\rightarrow
\Gamma(\slashed{S}^{-}_{\mathbb{C}}(X)\otimes EndE) \nonumber \end{equation}
following Theorem 13.10 of \cite{lawson}. $[ker(\slashed{D}_{A^{*}\otimes A})-coker(\slashed{D}_{A^{*}\otimes A})]$
exists as an element in the $K$-theory $K(pt)$ of one point, and there is a family version of the above construction
over the space $\widetilde{\mathcal{B}}_{X}$ of gauge equivalent classes of framed unitary connections on $E$.
The index bundle
\begin{equation} Ind(\slashed{\mathbb{D}}_{End\mathcal{E}})\in K(\widetilde{\mathcal{B}}_{X}), \nonumber \end{equation}
exists \cite{as4} whose determinant $\mathcal{L}_{\mathbb{C}}=det(Ind(\slashed{\mathbb{D}}_{End\mathcal{E}}))$
is a complex line bundle over $\widetilde{\mathcal{B}}_{X}$. This determinant line bundle has some remarkable properties depending on the
dimension of $X$. To explain that, we first recall the following standard facts about spin geometry (see also Theorem \ref{rep of spin gp}).
\begin{lemma}(\cite{adams}, \cite{lawson})
Let $\slashed{S}^{\pm}_{\mathbb{C}}(X)$ be the complex spinor bundles of an even dimensional spin manifold $X$. Then \\
(1) if $dimX=8k$, there exists real spinor bundles $\slashed{S}^{\pm}(X)$ such that
$\slashed{S}^{\pm}_{\mathbb{C}}(X)=\slashed{S}^{\pm}(X)\otimes_{\mathbb{R}}\mathbb{C}$; \\
(2) if $dimX=4k+2$, $\slashed{S}^{+}_{\mathbb{C}}(X)\cong(\slashed{S}^{-}_{\mathbb{C}}(X))^{*}$ as Clifford bundles.
\end{lemma}
From this lemma, we can obtain corresponding structures on $\mathcal{L}_{\mathbb{C}}$, i.e. \\
(1) if $dimX=8k$, there exists a real line bundle $\mathcal{L}_{\mathbb{R}}$ such that
$\mathcal{L}_{\mathbb{C}}\cong\mathcal{L}_{\mathbb{R}}\otimes_{\mathbb{R}}\mathbb{C}$.
In other words, there exists a non-degenerate quadratic form $Q$ on $\mathcal{L}_{\mathbb{C}}$ with
\begin{equation}Q: \mathcal{L}_{\mathbb{C}}\otimes\mathcal{L}_{\mathbb{C}}\cong \mathbb{C}\times\widetilde{\mathcal{B}}_{X}. \nonumber \end{equation}
(2) If $dimX=4k+2$, the extended determinant line bundle
$\mathcal{L}_{\mathbb{C}}\rightarrow\widetilde{\mathcal{B}}_{X}\times\widetilde{\mathcal{B}}_{X}$ (see Section 3)
satisfies\footnote{This follows from a brilliant idea due to
Maulik, Nekrasov and Okounkov \cite{no}.}
\begin{equation}\sigma^{*}\mathcal{L}_{\mathbb{C}}\cong\mathcal{L}_{\mathbb{C}}, \nonumber \end{equation}
where
\begin{equation}\sigma: \widetilde{\mathcal{B}}_{X}\times \widetilde{\mathcal{B}}_{X}\rightarrow
\widetilde{\mathcal{B}}_{X}\times \widetilde{\mathcal{B}}_{X}, \nonumber \end{equation}
\begin{equation}\sigma([A_{1}],[A_{2}])=([A_{2}],[A_{1}]). \nonumber \end{equation}
Furthermore, we will show the following orientability result.
\begin{theorem}(Theorem \ref{ori even}, Theorem \ref{ori odd}) ${}$ \\
Let $X$ be a compact spin manifold of even dimension, and $(E,h)$ be a Hermitian complex vector bundle. Then \\
(1) if $dimX=8k$, the structure group of $(\mathcal{L}_{\mathbb{C}},Q)$ can be reduced to $SO(1,\mathbb{C})$, i.e.
the corresponding real line bundle $\mathcal{L}_{\mathbb{R}}$ is trivial, provided that $H_{odd}(X,\mathbb{Z})=0$ ; \\
(2) if $dimX=4k+2$, $\mathcal{L}_{\mathbb{C}}$ has natural\footnote{See Theorem \ref{ori odd} for the precise meaning.}
choices of square roots parametrized by $Hom(H^{odd}(X,\mathbb{Z}),\mathbb{Z}_{2})$.
\end{theorem}
On Calabi-Yau manifolds, $\slashed{S}_{\mathbb{C}}(X)=\bigwedge^{0,*}(X)$ and $\slashed{D}=\overline{\partial}$, the above result gives an
orientability for (coarse) moduli spaces of simple holomorphic bundles. By a machinery (heavily used by
Joyce-Song \cite{js}) called Seidel-Thomas twist \cite{st}, we can extend it to moduli spaces of simple coherent sheaves.
\begin{theorem}(Theorem \ref{ori on even CY}, Theorem \ref{ori on odd CY}) ${}$ \\
Let $X$ be a projective Calabi-Yau $n$-fold with $Hol(X)=SU(n)$, $\mathcal{M}_{X}$ be
a coarse moduli space of simple sheaves with fixed Chern classes \footnote{We endow it with the induced complex analytic topology
as page 54 of \cite{js}. One could also impose the Gieseker stability condition to get a projective scheme as the moduli space \cite{hl}. },
and we denote its determinant line bundle by $\mathcal{L}_{\mathcal{M}_{X}}$.
%with $\mathcal{L}_{\mathcal{M}_{X}}|_{\mathcal{F}}\cong det(Ext^{odd}(\mathcal{F},\mathcal{F}))\otimes det(Ext^{even}(\mathcal{F},\mathcal{F}))^{-1}$.
Then, we have \\
(1) if $n=4k$, structure group of $(\mathcal{L}_{\mathcal{M}_{X}},Q_{Serre})$ can be reduced to $SO(1,\mathbb{C})$, when $H_{odd}(X,\mathbb{Z})=0$ \\
(2) if $n=4k+2$, the structure group of $(\mathcal{L}_{\mathcal{M}_{X}},Q_{Serre})$ is canonically reduced to $SO(1,\mathbb{C})$
\footnote{This is first observed by Borisov and Joyce \cite{bj}.}; \\
(3) if $n$ is odd, each element in $Hom(H^{odd}(X,\mathbb{Z}),\mathbb{Z}_{2})$ determines an (algebraic) square root of
$\mathcal{L}_{\mathcal{M}_{X}}|_{\mathcal{M}_{X}^{red}}$ over the reduced scheme $\mathcal{M}_{X}^{red}$, when $\mathcal{M}_{X}$ is a proper scheme.
\end{theorem}
%The proof of this main theorem relies on varies tricks in gauge theory: firstly we use a machinery (heavily used before by
%Joyce-Song \cite{js}) called Seidel-Thomas twist \cite{st} to transform the problem to a problem on some coarse moduli spaces of
%simple holomorphic bundles (with complex analytic topology); secondly, we use Donaldson's argument \cite{d1} to further
%reduce the problem to complex vector bundles with high ranks and vanishing first Chern classes; thirdly,
%we calculate the torsion part of the second cohomology group (as an abelian group) of the base of the index bundle and
%prove it vanishes; finally, we apply the Atiyah-Singer family index theorem to calculate Chern classes of the index bundle directly
%(modulo torsion). We remark that in the proof of part $(2)$ of the above theorem, we further need to use a brilliant idea due to
%Maulik, Nekrasov and Okounkov \cite{no}. \\
This result in fact fits into the work of Borisov and Joyce on the definition of
orientations for derived schemes with shifted symplectic structures (see Definition 2.11 of \cite{bj}) and Joyce's definition of
orientations for d-critical loci \cite{joyce-d} (used to categorify $DT_{3}$ invariants).
In general, derived moduli schemes of simple sheaves on Calabi-Yau $n$-folds are expected to
have $(2-n)$-shifted symplectic structures in the sense of Pantev, T\"{o}en, Vaqui\'{e} and Vezzosi \cite{ptvv}.
When $n$ is even, there is a canonical isomorphism
\begin{equation}Q_{Serre}: \mathcal{L}_{\mathcal{M}_{X}}\otimes\mathcal{L}_{\mathcal{M}_{X}}\cong \mathcal{O}_{\mathcal{M}_{X}} \nonumber \end{equation}
and the orientability issue is to find a square root of this isomorphism \cite{bj}.
When $n$ is odd, the orientability issue is to find a square root of $\mathcal{L}_{\mathcal{M}_{X}}|_{\mathcal{M}^{red}_{X}}$
over the reduced scheme $\mathcal{M}^{red}_{X}$ \cite{joyce-d}.

Along this line, we also prove an orientability result for the relative situation where we have Calabi-Yau manifolds as anti-canonical
divisors of even dimensional projective manifolds.
This will be useful in the relative $DT_{4}$ theory \cite{caoleung2}
(which is part of the complexification of Donaldson-Floer TQFT theory on 4-3 dimensional manifolds).
\begin{theorem}(Weak relative orientability, Theorem \ref{ori on rel}) ${}$ \\
Let $Y$ be a smooth anti-canonical divisor in a projective $2n$-fold $X$ with $Tor(H_{*}(X,\mathbb{Z}))=0$, $E\rightarrow X$ be a
complex vector bundle with structure group $SU(N)$, where $N\gg0$.
Let $\mathcal{M}_{X}$ be a coarse moduli scheme of simple holomorphic structures on $E$, which has a well-defined restriction morphism
\begin{equation}r: \mathcal{M}_{X}\rightarrow \mathcal{M}_{Y}, \nonumber \end{equation}
to a proper coarse moduli scheme of simple bundles on $Y$ with fixed Chern classes.

Then there exists an algebraic square root $(\mathcal{L}_{\mathcal{M}_{Y}}|_{\mathcal{M}^{red}_{Y}})^{\frac{1}{2}}$ of
$\mathcal{L}_{\mathcal{M}_{Y}}|_{\mathcal{M}^{red}_{Y}}$ such that
\begin{equation}c_{1}(\mathcal{L}_{\mathcal{M}_{X}}|_{\mathcal{M}^{red}_{X}})
=r^{*}c_{1}((\mathcal{L}_{\mathcal{M}_{Y}}|_{\mathcal{M}^{red}_{Y}})^{\frac{1}{2}}),
\nonumber \end{equation}
%\begin{equation}\mathcal{L}_{\mathcal{M}_{X}}\cong r^{*}(\mathcal{L}_{\mathcal{M}_{Y}}^{\frac{1}{2}}), \nonumber \end{equation}
where $\mathcal{L}_{\mathcal{M}_{X}}$ (resp. $\mathcal{L}_{\mathcal{M}_{Y}}$) is the determinant line bundle of $\mathcal{M}_{X}$
(resp. $\mathcal{M}_{Y}$).
\end{theorem}
Given such a restriction morphism $r$, $\mathcal{M}_{X}$ is expected to be a Lagrangian (see Calaque \cite{calaque}) of
the derived scheme $\mathcal{M}_{Y}$ with $(3-2n)$-shifted symplectic structure in the sense of Pantev, T\"{o}en, Vaqui\'{e} and Vezzosi \cite{ptvv}.
Then there is a canonical isomorphism
\begin{equation}(\mathcal{L}_{\mathcal{M}_{X}})^{\otimes2}\cong r^{*}\mathcal{L}_{\mathcal{M}_{Y}}   \nonumber \end{equation}
between determinant line bundles (it is verified directly in Lemma \ref{rel ori lemma}). The orientability issue in this relative case
is to find a square root of this isomorphism (see Definition \ref{def on rel ori}),
which is partially verified by the above weak relative orientability result and Proposition \ref{prop on rel ori}. \\
${}$ \\
\textbf{Content of the paper}: In section 2, we study the orientability issue for moduli spaces of simple sheaves on
Calabi-Yau manifolds of even dimensions. We first prove a general result for spin manifolds of $8k$ dimensions and then apply it to the case of
Calabi-Yau even-folds. Then we discuss its relation to the work of Borisov and Joyce \cite{bj} on the definition of orientations for derived schemes
with shifted symplectic structures \cite{ptvv}. In section 3, we study the orientability issue for
moduli spaces of simple sheaves on Calabi-Yau manifolds of odd dimensions (and corresponding results on spin manifolds of $(4k+2)$ dimensions).
The orientability result fits into Joyce's definition of orientations for d-critical loci \cite{joyce-d}
(used to categorify $DT_{3}$ invariants). In section 4, we discuss the orientability issue for the relative situation.
We define relative orientations for restriction morphisms and partially verify their existences.
In the appendix, we list some useful facts on spin geometry, gauge theory
and Seidel-Thomas twists. \\
${}$ \\
\textbf{Acknowledgement}:
The first author expresses his deep gratitude to Professor Simon Donaldson, Dominic Joyce and Tony Pantev for many useful discussions.
We thank Zheng Hua for varies helpful discussions and comments. The work of the second author was substantially supported by
grants from the Research Grants Council of the Hong Kong Special Administrative Region, China (Project No. CUHK401411 and CUHK14302714).

\section{Orientability for even dimensional Calabi-Yau }
%We move to even complex dimensional Calabi-Yau manifolds \footnote{We always mean $Hol(X)=SU(n)$.} and extend the story of
%orientations for moduli spaces of sheaves on Calabi-Yau 4-folds.
We fix a compact spin manifold $X$ of even dimension and a (Hermitian) complex vector bundle $(E,h)\rightarrow X$.
Given an unitary connection $A$ on $E$, we define the twisted Dirac operator
\begin{equation}\slashed{D}_{A^{*}\otimes A}: \Gamma(\slashed{S}^{+}_{\mathbb{C}}(X)\otimes EndE)\rightarrow
\Gamma(\slashed{S}^{-}_{\mathbb{C}}(X)\otimes EndE) \nonumber \end{equation}
following Theorem 13.10 of \cite{lawson}. $[ker(\slashed{D}_{A^{*}\otimes A})-coker(\slashed{D}_{A^{*}\otimes A})]$
exists as an element in the $K$-theory $K(pt)$ of one point, and there is a family version of the above construction
as follows\footnote{More details are explained in the appendix.}.

Let $\mathcal{A}$ be the space of all unitary connections on $(E,h)$, and $\mathcal{G}$ be the group of
unitary gauge transformations. We denote the $U(r)$-principal bundle (of frames) of $E$ by $P$, fix a base point $x_{0}\in X$ and introduce the space
\begin{equation}\widetilde{\mathcal{B}}_{X}=\mathcal{A}\times_{\mathcal{G}}P_{x_{0}} \quad (=\widetilde{\mathcal{B}}_{E,X}) \nonumber \end{equation}
of gauge equivalent classes of framed connections. Equivalently, $\widetilde{\mathcal{B}}_{X}=\mathcal{A}/\mathcal{G}_{0}$,
where $\mathcal{G}_{0}\vartriangleleft \mathcal{G}$ is the subgroup of gauge transformations which fix the fiber $P_{x_{0}}$.
As $\mathcal{G}_{0}$ acts freely on $\mathcal{A}$, $\widetilde{\mathcal{B}}_{X}$ (with suitable Sobolev structure)
has a Banach manifold structure whose weak homotopy type will not depend on the chosen Sobolev structures (Proposition 5.1.4 \cite{dk}).

Meanwhile, there exists a universal bundle $\mathcal{E}=\mathcal{A}\times_{\mathcal{G}_{0}}E$ over $\widetilde{\mathcal{B}}_{X}\times X$
(trivialized on $\widetilde{\mathcal{B}}_{X}\times\{x_{0}\}$), which carries a universal family of framed connections.
We then couple the Dirac operator $\slashed{D}$ on $X$ with the connection on $\mathcal{E}$ and there is an index bundle
\begin{equation} Ind(\slashed{\mathbb{D}}_{End\mathcal{E}})\in K(\widetilde{\mathcal{B}}_{X}), \nonumber \end{equation}
which satisfies $Ind(\slashed{\mathbb{D}}_{End\mathcal{E}})|_{[A]}=ker(\slashed{D}_{A^{*}\otimes A})-coker(\slashed{D}_{A^{*}\otimes A})\in K(pt)$
\cite{as4}. The determinant $\mathcal{L}_{\mathbb{C}}=det(Ind(\slashed{\mathbb{D}}_{End\mathcal{E}}))$ of
$Ind(\slashed{\mathbb{D}}_{End\mathcal{E}})$ exists as a complex line bundle over $\widetilde{\mathcal{B}}_{X}$. Meanwhile, the $U(r)$-action on
$\widetilde{\mathcal{B}}_{X}$ which changes framing at $P_{x_{0}}$ naturally extends to the line bundle
$\mathcal{L}_{\mathbb{C}}\rightarrow\widetilde{\mathcal{B}}_{X}$.

If the spin manifold $X$ is of real dimension $8k$, the complex spinor bundle $\slashed{S}^{\pm}_{\mathbb{C}}(X)$ is the
complexification of real spinor bundle $\slashed{S}^{\pm}(X)$, i.e. $\slashed{S}^{\pm}_{\mathbb{C}}(X)=\slashed{S}^{\pm}(X)\otimes_{\mathbb{R}}\mathbb{C}$
(see page 99 of \cite{lawson} or Theorem \ref{rep of spin gp}).
Then $\slashed{S}^{\pm}_{\mathbb{C}}(X)\otimes_{\mathbb{C}} EndE=(\slashed{S}^{\pm}(X)\otimes_{\mathbb{R}} \mathfrak{g}_{E})\otimes_{\mathbb{R}}\mathbb{C}$ and the corresponding
twisted Dirac operator is the complexification of the real one. Thus the determinant line bundle $\mathcal{L}_{\mathbb{C}}$ is the
complexification of a real determinant line bundle $\mathcal{L}_{\mathbb{R}}=det(Ind(\slashed{\mathbb{D}}_{\mathfrak{g}_{\mathcal{E}}}))$ for twisted Dirac operators of type
\begin{equation}\slashed{D}_{A^{*}\otimes A}: \Gamma(\slashed{S}^{+}(X)\otimes_{\mathbb{R}} \mathfrak{g}_{E})\rightarrow
\Gamma(\slashed{S}^{-}(X)\otimes_{\mathbb{R}} \mathfrak{g}_{E}). \nonumber \end{equation}
This then defines a non-degenerate quadratic form $Q$ on $\mathcal{L}_{\mathbb{C}}$ and gives a trivialization
\begin{equation}Q: \mathcal{L}_{\mathbb{C}}\otimes\mathcal{L}_{\mathbb{C}}\cong \mathbb{C}\times\widetilde{\mathcal{B}}_{X}. \nonumber \end{equation}
\begin{theorem}\label{ori even}
For any compact spin manifold $X$ of real dimension $8k$ with $H_{odd}(X,\mathbb{Z})=0$, and a Hermitian vector bundle $E\rightarrow X$, the
structure group of $(\mathcal{L}_{\mathbb{C}},Q)$ can be reduced to $SO(1,\mathbb{C})$, i.e.
%\begin{equation}c^{U(r)}_{1}(det(Ind(\slashed{\mathbb{D}}_{End\mathcal{E}})))=0, \nonumber \end{equation}
the corresponding real line bundle $\mathcal{L}_{\mathbb{R}}$ of $(\mathcal{L}_{\mathbb{C}},Q)$ is trivial.
\end{theorem}
\begin{proof}
Following the approach by Donaldson \cite{d1}, \cite{dk}, by considering $E^{'}=E\oplus (detE)^{-1}\oplus \mathbb{C}^{p}$, we have a stabilization map
\begin{equation}s: \widetilde{\mathcal{B}}_{E,X}\rightarrow \widetilde{\mathcal{B}}_{E',X}, \quad s(A)=A\oplus (det(A))^{*}\oplus\theta
\nonumber \end{equation}
%\begin{equation}s(A)=A\oplus det(A)^{-1}\oplus\theta, \nonumber \end{equation}
where $\theta$ is the rank $p$ product connection. When a $SU(N)$ connection on $E'$ decomposes as $A\oplus (det(A))^{*}\oplus\theta$, there is a
decomposition of the adjoint bundle
\begin{equation}\mathfrak{g}_{E'}=\mathfrak{g}_{E}\oplus V\oplus \mathfrak{g}_{\mathbb{C}^{p}}, \quad V\triangleq ((detE)\otimes E)\oplus ((\mathbb{C}^{p})^{*}\otimes E)\oplus ((detE)\otimes\mathbb{C}^{p}).  \nonumber \end{equation}
%\begin{equation}EndE'=EndE\oplus T^{*}V\oplus End(\mathbb{C}^{p}), \quad V\triangleq ((detE)\otimes E)\oplus ((\mathbb{C}^{p})^{*}\otimes E)\oplus ((detE)\otimes\mathbb{C}^{p})  \nonumber \end{equation}
The index of any operator coupled with this bundle (with connection) is a sum of corresponding terms.
In the obvious notations, we have
\begin{equation}s^{*}(det(Ind(\slashed{\mathbb{D}}_{\mathfrak{g}_{\mathcal{E}'}})))=det(Ind(\slashed{\mathbb{D}}_{\mathfrak{g}_{\mathcal{E}}}))\otimes
det(Ind(\slashed{\mathbb{D}}_{\mathbb{V}}))\otimes det(Ind(\slashed{\mathbb{D}}_{\mathfrak{g}_{\mathbb{C}^{p}}})).    \nonumber \end{equation}
As $V$ is complex, $det(Ind(\slashed{\mathbb{D}}_{\mathbb{V}}))$ has a canonical orientation. $det(Ind(\slashed{\mathbb{D}}_{\mathfrak{g}_{\mathbb{C}^{p}}}))$ is also trivial as $\mathbb{C}^{p}$ is a product bundle.
Then $s^{*}(det(Ind(\slashed{\mathbb{D}}_{\mathfrak{g}_{\mathcal{E}'}})))\cong det(Ind(\slashed{\mathbb{D}}_{\mathfrak{g}_{\mathcal{E}}}))$.

Then we are left to show $det(Ind(\slashed{\mathbb{D}}_{\mathfrak{g}_{\mathcal{E}}}))$ is trivial for a $SU(N)$ complex vector bundle $E$ on $X$ with $N\gg0$. Analogs to Theorem 10.14 of \cite{caoleung}, we apply the Federer spectral sequence \cite{mccleary},
\begin{equation}E_{2}^{p,q}\cong H^{p}(X,\pi_{p+q}(BSU(N)))\Rightarrow \pi_{q}(Map_{E}(X,BSU(N))).  \nonumber \end{equation}
For $N\gg0$, we get $\pi_{1}(Map_{E}(X,BSU(N)))\cong \bigoplus_{k\geq1} H^{2k+1}(X,\mathbb{Z})$, which vanishes by the assumption.
From Atiyah-Bott (Proposition 2.4 \cite{ab1}), we have a homotopy equivalence
\begin{equation}B\mathcal{G}\simeq Map_{E}(X,BSU(N)).  \nonumber \end{equation}
Then
\begin{equation}\pi_{1}(\widetilde{\mathcal{B}}_{X})\cong\pi_{1}((\mathcal{A}\times SU(N))/\mathcal{G})\cong\pi_{0}(\mathcal{G})\cong \pi_{1}(Map_{E}(X,BSU(N)))=0. \nonumber \end{equation}
Then any real line bundle over $\widetilde{\mathcal{B}}_{X}$ (including $\mathcal{L}_{\mathbb{R}}=det(Ind(\slashed{\mathbb{D}}_{\mathfrak{g}_{\mathcal{E}}}))$) is trivial.
%Meanwhile, $\pi_{1}(\widetilde{\mathcal{B}}_{X}\times_{SU(N)} ESU(N))\cong\pi_{1}(\widetilde{\mathcal{B}}_{X})$ and $H_{1}(\widetilde{\mathcal{B}}_{X}\times_{SU(N)} ESU(N),\mathbb{Z})\cong\pi_{1}(\widetilde{\mathcal{B}}_{X}\times_{SU(N)} ESU(N))$. Then $H^{2}(\widetilde{\mathcal{B}}_{X}\times_{SU(N)} ESU(N),\mathbb{Z})$ is also torsion-free.
%As $SU(N)$ preserves $det(Ind(\slashed{\mathbb{D}}_{End\mathcal{E}}))$, the line bundle lifts to $\widetilde{\mathcal{B}}_{X}\times_{SU(N)} ESU(N)$.
%We are thus left to show for any embedded surface $C\subseteq \widetilde{\mathcal{B}}_{X}\times_{SU(N)} ESU(N)$, $c_{1}(det(Ind(\slashed{\mathbb{D}}_{End\mathcal{E}}))|_{C})=0$. We denote $\mathcal{E}|_{C}$ to be the universal bundle on $X\times C$, and apply the Atiyah-Singer family index theorem \cite{as4},
%\begin{equation}c_{1}(det(Ind(\slashed{\mathbb{D}}_{End\mathcal{E}}))|_{C})=[ch(End \mathcal{E}|_{C})\cdot \hat{A}(X)]^{(4m+2)}/[X]=0,  \nonumber \end{equation}
%as $ch_{odd}(End \mathcal{E}|_{C})=0$.
\end{proof}
We fix a Calabi-Yau $4n$-fold $X$, and denote the determinant line bundle of a coarse moduli space $\mathcal{M}_{X}$ of
simple sheaves by $\mathcal{L}_{\mathcal{M}_{X}}$  with
$\mathcal{L}_{\mathcal{M}_{X}}|_{\mathcal{F}}\cong det(Ext^{odd}(\mathcal{F},\mathcal{F}))\otimes det(Ext^{even}(\mathcal{F},\mathcal{F}))^{-1}$.
The Serre duality pairing defines a non-degenerate quadratic form $Q_{Serre}$ on $\mathcal{L}_{\mathcal{M}_{X}}$ and gives a trivialization
\begin{equation}Q_{Serre}: \mathcal{L}_{\mathcal{M}_{X}}\otimes\mathcal{L}_{\mathcal{M}_{X}}\cong \mathcal{O}_{\mathcal{M}_{X}}. \nonumber \end{equation}
\begin{theorem}\label{ori on even CY}
Let $X$ be a projective Calabi-Yau $4n$-fold with $H_{odd}(X,\mathbb{Z})=0$, $\mathcal{M}_{X}$ be a coarse moduli space of simple sheaves with fixed Chern classes.

Then the structure group of $(\mathcal{L}_{\mathcal{M}_{X}},Q_{Serre})$ can be reduced to $SO(1,\mathbb{C})$, i.e.
the corresponding real line bundle $\mathcal{L}_{\mathbb{R}}$ of $(\mathcal{L}_{\mathcal{M}_{X}},Q_{Serre})$ is trivial.
In particular, $\mathcal{L}_{\mathcal{M}_{X}}\cong\mathcal{O}_{\mathcal{M}_{X}}$.
%$\mathcal{L}_{\mathcal{M}_{X}}\cong \mathcal{O}_{\mathcal{M}_{X}}$.
\end{theorem}
\begin{proof}
By the work of Joyce-Song \cite{js} (see Corollary \ref{st preserves L} in the appendix), we are reduced to consider the case when $\mathcal{M}_{X}$ is
a coarse moduli space of simple holomorphic structures on a complex bundle $E$ of rank $r$.

Let $\mathcal{A}^{*}\subseteq \mathcal{A}$ be the subspace of irreducible unitary connections on $(E,h)$,
$\widetilde{\mathcal{B}}_{X}^{*}=\mathcal{A}^{*}/\mathcal{G}_{0}\subseteq\widetilde{\mathcal{B}}_{X}$
be the open subset of irreducible framed connections whose complement is of infinite codimension (page 181 of \cite{dk}).
As Section 9.1 of \cite{js}, we introduce $\mathcal{A}^{(0,1)}_{si}$ to be the space of simple $(0,1)$-connections on $E$.
There is a group $\mathcal{G}^{c}$ of complex gauge transformations acting on $\mathcal{A}^{(0,1)}_{si}$ with stabilizer
$\mathbb{C}^{*}\cdot Id_{E}$. The subgroup $\mathcal{G}^{c}_{0}$ which preserves a fiber $E_{x_{0}}$ then acts freely on $\mathcal{A}^{(0,1)}_{si}$,
and $\mathcal{A}^{(0,1)}_{si}/\mathcal{G}^{c}_{0}$ is a Banach complex manifold (with suitable Banach completions, see \cite{js}).
Via the Hermitian metric $h$, $\mathcal{A}^{(0,1)}_{si}\cong\mathcal{A}^{*}$,
we then have an embedding $\mathcal{A}^{(0,1)}_{si}/\mathcal{G}^{c}_{0}\subseteq \widetilde{\mathcal{B}}_{X}^{*}$.
There is also a forgetful map (similar to (5.1.3) in \cite{dk})
\begin{equation}\beta: \mathcal{A}^{(0,1)}_{si}/\mathcal{G}^{c}_{0}\rightarrow\mathcal{A}^{(0,1)}_{si}/\mathcal{G}^{c}_{red},
\nonumber \end{equation}
which is a principal $PGL(r,\mathbb{C})$-bundle, where $\mathcal{G}^{c}_{red}=\mathcal{G}^{c}/\mathbb{C}^{*}$ and
$\mathbb{C}^{*}\subseteq\mathcal{G}^{c}$ is the subgroup of multiples of the identity map (see page 133 of \cite{js}).

Our coarse moduli space $\mathcal{M}_{X}$ of simple holomorphic bundles then sits inside $\mathcal{A}^{(0,1)}_{si}/\mathcal{G}^{c}_{red}$ as component(s) of integrable connections. As Calabi-Yau $4n$-folds are spin manifolds of dimensions $8m$, there is a quadratic line bundle $(\mathcal{L}_{\mathbb{C}},Q)$ on $\widetilde{\mathcal{B}}_{X}$ from the previous discussion. Its pull-back to $\mathcal{A}^{(0,1)}_{si}/\mathcal{G}^{c}_{0}$ via the imbedding
\begin{equation}\mathcal{A}^{(0,1)}_{si}/\mathcal{G}^{c}_{0}\subseteq\widetilde{\mathcal{B}}_{X}^{*}\subseteq\widetilde{\mathcal{B}}_{X}
\nonumber \end{equation}
is $GL(r,\mathbb{C})$ invariant (changing the framing at $E_{x_{0}}$).
So it descends to a quadratic line bundle over $\mathcal{A}^{(0,1)}_{si}/\mathcal{G}^{c}_{red}$ (via $\beta$) (see also 5.4.2 of \cite{dk}).
On Calabi-Yau manifolds $X$'s, $\slashed{S}_{\mathbb{C}}(X)=\bigwedge^{0,*}(X)$ and
$\slashed{D}=\overline{\partial}$, the descended quadratic line bundle (over $\mathcal{A}^{(0,1)}_{si}/\mathcal{G}^{c}_{red}$)
pulls back to the determinant line bundle $\mathcal{L}_{\mathcal{M}_{X}}$ with $Q_{Serre}$ over $\mathcal{M}_{X}$.
By Theorem \ref{ori even}, the real line bundle $\mathcal{L}_{\mathbb{R}}$ associated to $(\mathcal{L}_{\mathbb{C}},Q)$
is trivial over $\mathcal{A}^{(0,1)}_{si}/\mathcal{G}^{c}_{0}$. As the fiber of the map $\beta$ is connected, $\mathcal{L}_{\mathbb{R}}$ descends to
a trivial line bundle over $\mathcal{A}^{(0,1)}_{si}/\mathcal{G}^{c}_{red}$, in particular over $\mathcal{M}_{X}$.
%and the Serre duality pairing defines a non-degenerate quadratic form on the index bundle $Ind(\slashed{\mathbb{D}}_{End\mathcal{E}})$. Then $c_{1}(Ind(\slashed{\mathbb{D}}_{End\mathcal{E}}))=0$ implies the structure group is reduced to $SO(\bullet,\mathbb{C})$, i.e. $\mathcal{L}_{\mathcal{M}_{X}}\cong \mathcal{O}_{\mathcal{M}_{X}}$.
\end{proof}
\begin{remark}
$H_{odd}(X,\mathbb{Z})=0$ holds true for complete intersections $X$'s in smooth toric varieties \cite{fulton1}.
\end{remark}
Given a Calabi-Yau $2n$-fold $X$, and a coarse moduli space $\mathcal{M}_{X}$ of simple sheaves, the Serre duality pairing
gives a non-degenerate quadratic form on the
determinant line bundle $\mathcal{L}_{\mathcal{M}_{X}}$, which defines an isomorphism
\begin{equation}\label{equ 1}\mathcal{L}_{\mathcal{M}_{X}}\otimes\mathcal{L}_{\mathcal{M}_{X}}\cong \mathcal{O}_{\mathcal{M}_{X}}.  \end{equation}
The above Theorem \ref{ori on even CY} in fact shows that we can find an isomorphism $\mathcal{L}_{\mathcal{M}_{X}}\cong\mathcal{O}_{\mathcal{M}_{X}}$
whose square is the above given isomorphism.

This fits into the work of Borisov and Joyce on the orientation of derived schemes with $k$-shifted symplectic structure for even $k\leq0$
(see Definition 2.11 of \cite{bj}). In general, derived moduli schemes of simple sheaves on Calabi-Yau $n$-folds are expected to
have $(2-n)$-shifted symplectic structures in the sense of Pantev, T\"{o}en, Vaqui\'{e} and Vezzosi \cite{ptvv}.
When $n$ is even, there is a canonical isomorphism like (\ref{equ 1}).
The orientability issue in this case is to find a square root of this isomorphism (see \cite{bj} for more details).

The above Theorem \ref{ori on even CY} partially solves this issue for
Calabi-Yau $4n$-folds (which correspond to $k=(2-4n)$-shifted cases). In fact, for Calabi-Yau $(4n+2)$-folds (i.e. $k\equiv0$ mod $4$),
the determinant line bundle has a canonical trivialization \cite{bj}. To explain it in a simple way, we take a simple sheaf $\mathcal{F}$,
the determinant line bundle $\mathcal{L}_{\mathcal{M}_{X}}$ satisfies
\begin{equation}\mathcal{L}_{\mathcal{M}_{X}}|_{\mathcal{F}}\cong det(Ext^{2n+1}(\mathcal{F},\mathcal{F})),  \nonumber \end{equation}
as other terms are of type $det(V\oplus V^{*})$ and have canonical trivializations. The Serre duality pairing defines a non-degenerate $2$-form
(instead of a quadratic form) on $Ext^{2n+1}(\mathcal{F},\mathcal{F})$ (see also Theorem \ref{rep of spin gp} at the level of
spin representations), which gives a canonical trivialization of $det(Ext^{2n+1}(\mathcal{F},\mathcal{F}))$ as in the
holomorphic symplectic case (see Mukai \cite{mukai}).
\begin{remark}
Index bundles could be understood as tangent bundles of moduli spaces in the derived sense. By Theorem \ref{ori on even CY} and the above discussion,
we can regard moduli spaces of simple sheaves on Calabi-Yau $2n$-folds as 'derived' Calabi-Yau spaces.
\end{remark}

\section{Orientability for odd dimensional Calabi-Yau }
In \cite{no}, Nekrasov and Okounkov gave a short proof of the existence of square roots of determinant line bundles for moduli spaces of sheaves
on Calabi-Yau $(2n+1)$-folds. However, sometimes it would be useful to make the square root compatible with other structures,
such as the wall-crossing structure in the sense of Kontsevich and Soibelman \cite{ks} on moduli spaces. In \cite{hua1}, Hua showed that finding
(wall-crossing compatible) square roots is related to find square roots of determinant line bundles over spaces of
gauge equivalent classes of connections. Following the argument of Donaldson \cite{dk}, he then used geometric transitions
to prove the existence of square roots for simply-connected torsion-free Calabi-Yau $3$-folds.

In this section, we will show that the existence of square roots is in fact a phenomenon in spin geometry, and
prove their existence over spaces of gauge equivalent classes of irreducible connections on spin manifolds $X$'s with
$dim_{\mathbb{R}}(X)=8k+2$ or $8k+6$.
\begin{theorem}\label{ori odd}
Let $X$ be a compact spin manifold of real dimension $8k+2$ or $8k+6$, $(E,h)\rightarrow X$ be a (Hermitian) complex vector bundle, and
$N\geq rk(E)+1$ be a positive integer.

Then there exists a $SU(N)$ complex vector bundle $E'$ and a continuous map
$s:\widetilde{\mathcal{B}}_{E,X}\rightarrow\widetilde{\mathcal{B}}_{E',X}$ such that the quotient of determinant line bundles
\begin{equation}\frac{det(Ind(\slashed{\mathbb{D}}_{End\mathcal{E}}))}{s^{*}det(Ind(\slashed{\mathbb{D}}_{End\mathcal{E'}}))} \nonumber \end{equation}
has a canonical square root.

Furthermore, if $N\gg0$, $det(Ind(\slashed{\mathbb{D}}_{End\mathcal{E'}}))$ has square roots whose choices are parametrized by
$Hom(H^{odd}(X,\mathbb{Z}),\mathbb{Z}_{2})$.
\end{theorem}
\begin{proof}
As in the proof of Theorem \ref{ori even}, by considering $E^{'}=E\oplus (detE)^{-1}\oplus \mathbb{C}^{p}$, we have a stabilization map
\begin{equation}s: \widetilde{\mathcal{B}}_{E,X}\rightarrow \widetilde{\mathcal{B}}_{E',X}, \quad s(A)=A\oplus (det(A))^{*}\oplus\theta
\nonumber \end{equation}
%\begin{equation}s(A)=A\oplus det(A)^{-1}\oplus\theta, \nonumber \end{equation}
where $\theta$ is the rank $p$ product connection. When a $SU(N)$ connection on $E'$ decomposes as $A\oplus (det(A))^{*}\oplus\theta$, there is a
decomposition of the endomorphism bundle
%\begin{equation}\mathfrak{g}_{E'}=\mathfrak{g}_{E}\oplus V\oplus \mathfrak{g}_{\mathbb{C}^{p}}, \quad V\triangleq ((detE)\otimes E)\oplus ((\mathbb{C}^{p})^{*}\otimes E)\oplus ((detE)\otimes\mathbb{C}^{p}).  \nonumber \end{equation}
\begin{equation}EndE'=EndE\oplus T^{*}V\oplus End(\mathbb{C}^{p}), \quad V\triangleq ((detE)\otimes E)\oplus
((\mathbb{C}^{p})^{*}\otimes E)\oplus ((detE)\otimes\mathbb{C}^{p}).  \nonumber \end{equation}
The index of any operator coupled with this bundle (with connection) is a sum of corresponding terms.
In the obvious notation, we have
\begin{equation}s^{*}(det(Ind(\slashed{\mathbb{D}}_{End\mathcal{E}'})))=det(Ind(\slashed{\mathbb{D}}_{End\mathcal{E}}))\otimes
det(Ind(\slashed{\mathbb{D}}_{T^{*}\mathbb{V}}))\otimes det(Ind(\slashed{\mathbb{D}}_{End\mathbb{C}^{p}})).    \nonumber \end{equation}
By Corollary \ref{cor}, $det(Ind(\slashed{\mathbb{D}}_{T^{*}\mathbb{V}}))\cong
(det(Ind(\slashed{\mathbb{D}}_{\mathbb{V}})))^{\otimes2}$, which has a canonical square root.
$det(Ind(\slashed{\mathbb{D}}_{End\mathbb{C}^{p}})$ is canonically trivial as $\mathbb{C}^{p}$ is a product bundle.
Then
\begin{equation}s^{*}(det(Ind(\slashed{\mathbb{D}}_{End\mathcal{E}'})))\cong det(Ind(\slashed{\mathbb{D}}_{End\mathcal{E}}))\otimes
(det(Ind(\slashed{\mathbb{D}}_{\mathbb{V}})))^{\otimes2}. \nonumber \end{equation}
Following Nekrasov and Okounkov \cite{no}, we consider an involution map
\begin{equation}\sigma: \widetilde{\mathcal{B}}_{E',X}\times \widetilde{\mathcal{B}}_{E',X}\rightarrow
\widetilde{\mathcal{B}}_{E',X}\times \widetilde{\mathcal{B}}_{E',X}, \nonumber \end{equation}
\begin{equation}\sigma([A_{1}],[A_{2}])=([A_{2}],[A_{1}]).
\nonumber \end{equation}
We denote the extended determinant line bundle to be $\mathcal{L}=det(Ind(\slashed{\mathbb{D}}_{End\mathcal{E'}}))\rightarrow
\widetilde{\mathcal{B}}_{E',X}\times \widetilde{\mathcal{B}}_{E',X}$.
By applying the canonical isomorphism in Theorem \ref{thm 8k+2} to $EndE$, we can obtain
\begin{equation}\label{equality 2}\sigma^{*}\mathcal{L}\cong\mathcal{L}.
\end{equation}
Now we are reduced to prove $c_{1}(\mathcal{L}|_{\Delta})\equiv0 \textrm{ }(mod \textrm{ }2)$, where
$\Delta\hookrightarrow \widetilde{\mathcal{B}}_{E',X}\times \widetilde{\mathcal{B}}_{E',X}$ is the diagonal.

By the K\"{u}nneth formula,
\begin{equation}H^{2}(\widetilde{\mathcal{B}}_{E',X}\times \widetilde{\mathcal{B}}_{E',X},\mathbb{Z}_{2})\cong
H^{0}(\widetilde{\mathcal{B}}_{E',X},\mathbb{Z}_{2})\otimes H^{2}(\widetilde{\mathcal{B}}_{E',X},\mathbb{Z}_{2})\oplus  \nonumber \end{equation}
\begin{equation}
\oplus H^{2}(\widetilde{\mathcal{B}}_{E',X},\mathbb{Z}_{2})\otimes H^{0}(\widetilde{\mathcal{B}}_{E',X},\mathbb{Z}_{2})\oplus
H^{1}(\widetilde{\mathcal{B}}_{E',X},\mathbb{Z}_{2})\otimes H^{1}(\widetilde{\mathcal{B}}_{E',X},\mathbb{Z}_{2}).      \nonumber \end{equation}
Assume $\{a_{i}\}$ is a basis of $H^{0}(\widetilde{\mathcal{B}}_{E',X},\mathbb{Z}_{2})$, $\{b_{i}\}$ is a basis of $H^{2}(\widetilde{\mathcal{B}}_{E',X},\mathbb{Z}_{2})$,
$\{c_{i}\}$ is a basis of $H^{1}(\widetilde{\mathcal{B}}_{E',X},\mathbb{Z}_{2})$, and
\begin{equation}c_{1}(\mathcal{L})\equiv\sum_{i,j}n_{ij}a_{i}\otimes b_{j}+ \sum_{i,j}m_{ij}b_{i}\otimes a_{j}+ \sum_{i,j}k_{ij}c_{i}\otimes c_{j}\textrm{ } (mod \textrm{ }2). \nonumber \end{equation}
Under the action of the involution map $\sigma$,
\begin{equation}\sigma^{*}\big(c_{1}(\mathcal{L})\big)\equiv\sum_{i,j}m_{ij}a_{j}\otimes b_{i}+ \sum_{i,j}n_{ij}b_{j}\otimes a_{i}+
\sum_{i,j}k_{ij}c_{j}\otimes c_{i} \textrm{ } (mod \textrm{ }2).\nonumber \end{equation}
By (\ref{equality 2}), we obtain $m_{ji}\equiv n_{ij}\textrm{ }(mod \textrm{ }2)$, $k_{ji}\equiv k_{ij}\textrm{ }(mod \textrm{ }2)$. When we restrict to the diagonal,
%\begin{equation}w_{1}(\mathcal{L})=\sum_{i}n_{ij}(a_{ij}\otimes b_{j}+b_{i}\otimes a_{i})  \nonumber \end{equation}
\begin{equation}c_{1}(\mathcal{L}|_{\Delta})\equiv\sum_{i,j}n_{ij}(a_{i}\cup b_{j}+b_{j}\cup a_{i})\equiv0 \textrm{ }(mod \textrm{ }2).
\nonumber \end{equation}
If $N\gg0$, we have $H_{1}(\widetilde{\mathcal{B}}_{E',X},\mathbb{Z})\cong H^{odd}(X,\mathbb{Z})$ as showed in Theorem \ref{ori even}.
Meanwhile, the choice of square roots of any complex line bundle over a space $W$ is parametrized by $H^{1}(W,\mathbb{Z}_{2})$.
\end{proof}
We fix a Calabi-Yau $(2n+1)$-fold $X$, and denote the determinant line bundle of a coarse moduli space $\mathcal{M}_{X}$ of simple sheaves by
$\mathcal{L}_{\mathcal{M}_{X}}$  with $\mathcal{L}_{\mathcal{M}_{X}}|_{\mathcal{F}}\cong
det(Ext^{odd}(\mathcal{F},\mathcal{F}))\otimes det(Ext^{even}(\mathcal{F},\mathcal{F}))^{-1}$.
\begin{theorem}\label{ori on odd CY}
Let $X$ be a projective Calabi-Yau $(2n+1)$-fold, $\mathcal{M}_{X}$ be a proper coarse moduli scheme of simple sheaves with fixed Chern classes.

Then there is a $1$-$1$ correspondence between the set of principal $\mathbb{Z}_{2}$-bundles on $\mathcal{M}_{X}$ and the set of
algebraic square roots of
$\mathcal{O}_{\mathcal{M}^{red}_{X}}$. Moreover,
each element in $Hom(H^{odd}(X,\mathbb{Z}),\mathbb{Z}_{2})$ determines an algebraic square root of
$\mathcal{L}_{\mathcal{M}_{X}}|_{\mathcal{M}_{X}^{red}}$ over the reduced scheme $\mathcal{M}_{X}^{red}$.
\end{theorem}
\begin{proof}
As in the proof of Theorem \ref{ori on even CY}, using Theorem \ref{ori odd}, each element in $Hom(H^{odd}(X,\mathbb{Z}),\mathbb{Z}_{2})$
determines a (topological) square root of $\mathcal{L}_{\mathcal{M}_{X}}$.
By Lemma 6.1 \cite{no}, $\mathcal{L}_{\mathcal{M}_{X}}|_{\mathcal{M}_{X}^{red}}$ has an algebraic square root as
$c_{1}(\mathcal{L}_{\mathcal{M}_{X}})$ is even.
To show those (topological) square roots are algebraic square roots over $\mathcal{M}_{X}^{red}$,
we are left to show any principal $\mathbb{Z}_{2}$-bundle
is an algebraic square root of $\mathcal{O}_{\mathcal{M}_{X}^{red}}$.

From the short exact sequence
\begin{equation} \xymatrix@1{1\rightarrow\mathbb{Z}_{2}\rightarrow\mathcal{O}^{*}_{\mathcal{M}_{X}^{red}}\ar[r]^{\quad f\mapsto f^{2}} &
\mathcal{O}^{*}_{\mathcal{M}_{X}^{red}}\rightarrow 1 }, \nonumber \end{equation}
we obtain an exact sequence
\begin{equation}\xymatrix@1{0\rightarrow H^{0}(\mathcal{M}_{X}^{red},\mathbb{Z}_{2})\rightarrow H^{0}(\mathcal{M}_{X}^{red},\mathcal{O}^{*}_{\mathcal{M}_{X}^{red}})
\ar[r]^{\quad \quad\quad \quad f\mapsto f^{2}} & H^{0}(\mathcal{M}_{X}^{red},\mathcal{O}^{*}_{\mathcal{M}_{X}^{red}})\rightarrow  } \nonumber \end{equation}
\begin{equation}\xymatrix@1{\rightarrow H^{1}(\mathcal{M}_{X}^{red},\mathbb{Z}_{2})\ar[r]^{i} &
H^{1}(\mathcal{M}_{X}^{red},\mathcal{O}^{*}_{\mathcal{M}_{X}^{red}}) \ar[r]^{L\mapsto L^{2}\quad } &
H^{1}(\mathcal{M}_{X}^{red},\mathcal{O}^{*}_{\mathcal{M}_{X}^{red}})\rightarrow\cdot\cdot\cdot }.   \nonumber \end{equation}
As $\mathcal{M}_{X}^{red}$ is proper, any algebraic function on it is locally constant (ref. 10.3.7 of \cite{vakil}). So the above sequence splits and
the map $H^{1}(\mathcal{M}_{X}^{red},\mathbb{Z}_{2})\rightarrow H^{1}(\mathcal{M}_{X}^{red},\mathcal{O}^{*}_{\mathcal{M}_{X}})$ is injective.
Using the remaining exact sequence, it is obvious that the corresponding algebraic line bundle (image under map $i$)
is a square root of $\mathcal{O}_{\mathcal{M}_{X}^{red}}$.
\end{proof}
\begin{remark}
Based on Joyce's definition of orientations for d-critical locus \cite{joyce-d}, we only need square roots of determinant line bundles
over the reduced moduli schemes to categorify $DT_{3}$ invariants \cite{bbdjs}.
\end{remark}
The following example shows that one could not expect to get vanishing of first Chern classes of determinant line bundles for odd dimensional $CY$.
\begin{example}
Let $X$ be a generic quintic 3-fold, and consider the Hilbert scheme of two points on $X$ (which is smooth), i.e.
$Hilb^{(2)}(X)=Bl_{\Delta}(X\times X)/\mathbb{Z}_{2}$, where $\Delta\hookrightarrow X\times X$ is the diagonal.
Its determinant line bundle satisfies $c_{1}(\mathcal{L}_{Hilb^{(2)}(X)})=2c_{1}(Hilb^{(2)}(X))\neq0$.
\end{example}

\section{Orientability for the relative case }

\subsection{The weak orientability result}
We take a smooth (Calabi-Yau) $(2n-1)$-fold $Y$ in a smooth complex projective $2n$-fold $X$ as its anti-canonical divisor.
We denote $\mathcal{M}_{X}$ to be a coarse moduli space of simple bundles on $X$ with fixed Chern classes which has a well-defined restriction morphism
\begin{equation}r: \mathcal{M}_{X}\rightarrow \mathcal{M}_{Y}, \nonumber \end{equation}
to a coarse moduli space of simple bundles on $Y$ with fixed Chern classes.

The corresponding restriction morphism between reduced schemes\footnote{It is uniquely determined by $r$, see Ex. 2.3 \cite{hart}.} is still denoted by
\begin{equation}r: \mathcal{M}_{X}^{red}\rightarrow \mathcal{M}_{Y}^{red}. \nonumber \end{equation}
By Theorem \ref{ori on odd CY}, there exists square roots of $\mathcal{L}_{\mathcal{M}_{Y}}|_{\mathcal{M}^{red}_{Y}}$ coming from the restriction of
square roots of the determinant line bundle of the index bundle of twisted Dirac operators over the space of connections.
The following relative orientability result for the morphism $r$ gives an identification of complex line bundles
$r^{*}(\mathcal{L}_{\mathcal{M}_{Y}}|_{\mathcal{M}^{red}_{Y}})^{\frac{1}{2}}\cong\mathcal{L}_{\mathcal{M}_{X}}|_{\mathcal{M}^{red}_{X}}$.
\begin{theorem}(Weak relative orientability)\label{ori on rel} ${}$ \\
Let $Y$ be a smooth anti-canonical divisor in a projective $2n$-fold $X$ with $Tor(H_{*}(X,\mathbb{Z}))=0$, $E\rightarrow X$ be a
complex vector bundle with structure group $SU(N)$, where $N\gg0$.
Let $\mathcal{M}_{X}$ be a coarse moduli scheme of simple holomorphic structures on $E$, which has a well-defined restriction morphism
\begin{equation}r: \mathcal{M}_{X}\rightarrow \mathcal{M}_{Y}, \nonumber \end{equation}
to a proper coarse moduli scheme of simple bundles on $Y$ with fixed Chern classes.

Then there exists an algebraic square root $(\mathcal{L}_{\mathcal{M}_{Y}}|_{\mathcal{M}^{red}_{Y}})^{\frac{1}{2}}$ of
$\mathcal{L}_{\mathcal{M}_{Y}}|_{\mathcal{M}^{red}_{Y}}$ such that
\begin{equation}c_{1}(\mathcal{L}_{\mathcal{M}_{X}}|_{\mathcal{M}^{red}_{X}})
=r^{*}c_{1}((\mathcal{L}_{\mathcal{M}_{Y}}|_{\mathcal{M}^{red}_{Y}})^{\frac{1}{2}}),
\nonumber \end{equation}
%\begin{equation}\mathcal{L}_{\mathcal{M}_{X}}\cong r^{*}(\mathcal{L}_{\mathcal{M}_{Y}}^{\frac{1}{2}}), \nonumber \end{equation}
where $\mathcal{L}_{\mathcal{M}_{X}}$ (resp. $\mathcal{L}_{\mathcal{M}_{Y}}$) is the determinant line bundle of $\mathcal{M}_{X}$ (resp. $\mathcal{M}_{Y}$).
\end{theorem}
\begin{proof}
%From Lemma \ref{def-obs LES}, we get $(\mathcal{L}_{\mathcal{M}_{X}}\otimes r^{*}\mathcal{L}_{\mathcal{M}_{Y}}^{-\frac{1}{2}})^{2}\cong \mathcal{O}_{_{\mathcal{M}_{X}}}$, which defines a non-degenerate quadratic form on $\mathcal{L}_{\mathcal{M}_{X}}\otimes r^{*}\mathcal{L}_{\mathcal{M}_{Y}}^{-\frac{1}{2}}$. As in Theorem \ref{ori on even CY}, we then only need to prove
%\begin{equation}c_{1}(\mathcal{L}_{\mathcal{M}_{X}}\otimes r^{*}\mathcal{L}_{\mathcal{M}_{Y}}^{-\frac{1}{2}})=0. \nonumber \end{equation}
%By the Lefschetz hyperplane theorem, $Tor(H_{*}(Y,\mathbb{Z}))=0$.
Without loss of generality, we assume $\mathcal{M}_{X}$ and $\mathcal{M}_{Y}$ are reduced schemes.
Using a Hermitian metric on $E$, we have embeddings $\mathcal{M}_{X}\subseteq \mathcal{B}^{*}_{X}$, $\mathcal{M}_{Y}\subseteq \mathcal{B}^{*}_{Y}$
into spaces of gauge equivalent classes of irreducible unitary connections on $E$ and $E|_{Y}$.
The restriction map $r: \mathcal{M}_{X}\rightarrow\mathcal{M}_{Y}$ extends to a restriction map
$r: \mathcal{B}^{*}_{X}\rightarrow \mathcal{B}_{Y}$ to the orbit space of connections on $E|_{Y}$.
We consider an open subset $U(\mathcal{M}_{X})\triangleq r^{-1}(\mathcal{B}^{*}_{Y})$ of $\mathcal{B}^{*}_{X}$ which fits into a commutative diagram
\begin{equation}\label{diagram1}
\xymatrix{
  \beta^{-1}(U(\mathcal{M}_{X}))\ar[d]_{\beta} \ar[r]^{\quad \quad r}
                & \widetilde{\mathcal{B}}^{*}_{Y} \ar[d]^{\beta} \\
  U(\mathcal{M}_{X})   \ar[r]^{r}
                & \mathcal{B}^{*}_{Y}, }
\end{equation}
where $\beta$ is a $PSU(N)$-fiber bundle defined by forgetting the framing at a base point $x_{0}\in Y\subseteq X$ (see also (5.1.3) of \cite{dk}).

By Theorem \ref{ori on odd CY}, there exists an algebraic square root $\mathcal{L}_{\mathcal{M}_{Y}}^{\frac{1}{2}}$
coming from a square root $det(Ind(\slashed{\mathbb{D}}_{End\mathcal{E}_{Y}}))^{\frac{1}{2}}\rightarrow \widetilde{\mathcal{B}}_{Y}$.
In fact, its restriction to $\widetilde{\mathcal{B}}^{*}_{Y}\subseteq\widetilde{\mathcal{B}}_{Y}$ descends to a line bundle
over $\mathcal{B}^{*}_{Y}$ via map $\beta$, which
gives $\mathcal{L}_{\mathcal{M}_{Y}}^{\frac{1}{2}}$ on $\mathcal{M}_{Y}\subseteq \mathcal{B}^{*}_{Y}$.
Thus to prove the result for Chern classes of line bundles over $\mathcal{B}^{*}_{X}$,
we are left to prove a result for $G=PSU(N)$-equivariant Chern classes over the space $\widetilde{\mathcal{B}}^{*}_{X}$, i.e.
%As $\mathcal{B}_{Y}\backslash\mathcal{B}^{*}_{Y}\subseteq \mathcal{B}_{Y}$ has large codimension,
%$det(Ind(\slashed{\mathbb{D}}_{End\mathcal{E}_{Y}}))^{\frac{1}{2}}$ uniquely extends to a complex line bundle
%$det(Ind(\slashed{\mathbb{D}}_{End\mathcal{E}_{Y}}))^{\frac{1}{2}}\rightarrow\mathcal{B}_{Y}$.
\begin{equation}\label{ori of amb space}
c_{1}^{G}(det(Ind(\slashed{\mathbb{D}}_{End\mathcal{E}_{X}})))-c_{1}^{G}(r^{*}det(Ind(\slashed{\mathbb{D}}_{End\mathcal{E}_{Y}}))^{\frac{1}{2}})=0\in
H^{2}(\widetilde{\mathcal{B}}^{*}_{X}\times_{G}EG,\mathbb{Z}),  \end{equation}
for the $SU(N)$ complex vector bundle $E\rightarrow X$ with $N\gg0$,
where $r: \widetilde{\mathcal{B}}_{X}^{*}\rightarrow \widetilde{\mathcal{B}}_{Y}$ is the restriction map which extends the one in (\ref{diagram1}),
and $\mathcal{E}_{X}$ (resp. $\mathcal{E}_{Y}$) is the universal family over $\widetilde{\mathcal{B}}_{X}^{*}$ (resp. $\widetilde{\mathcal{B}}_{Y}^{*}$).
The index bundle $Ind(\slashed{\mathbb{D}}_{End\mathcal{E}_{X}})$ is defined by a lifting $c_{1}(X)$ of
$w_{2}(X)$ using the $spin^{c}$ structure of complex manifold $X$.
%satisfies $Ind(\slashed{\mathbb{D}}_{End\mathcal{E}_{X}})|_{A}=H^{0,odd}_{D_{A}}(X,EndE)-H^{0,even}_{D_{A}}(X,EndE)$.

The Federer spectral sequence
\begin{equation}E_{2}^{p,q}\cong H^{p}(X,\pi_{p+q}(BSU(N)))\Rightarrow \pi_{q}(Map_{E}(X,BSU(N))).  \nonumber \end{equation}
gives $\pi_{1}(Map_{E}(X,BSU(N)))\cong \bigoplus_{k\geq1} H^{2k+1}(X,\mathbb{Z})$ for $N\gg0$, which is torsion-free.
Since $G$ acts on $\widetilde{\mathcal{B}}_{X}^{*}\times EG$ freely, we have an exact sequence
\begin{equation}\pi_{1}(G)\rightarrow\pi_{1}(\widetilde{\mathcal{B}}_{X}^{*}\times EG)\rightarrow
\pi_{1}(\widetilde{\mathcal{B}}_{X}^{*}\times_{G}EG)\rightarrow 0. \nonumber \end{equation}
As $\pi_{1}(G)\cong \mathbb{Z}_{n}$, and $\pi_{1}(\widetilde{\mathcal{B}}_{X}^{*}\times EG)\cong
\pi_{1}(\widetilde{\mathcal{B}}_{X})\cong\pi_{1}(Map_{E}(X,BSU(N)))$ is torsion-free,
there is no homomorphism from torsion groups to torsion-free ones,
thus $\pi_{1}(\widetilde{\mathcal{B}}_{X}^{*}\times_{G}EG)\cong\pi_{1}(\widetilde{\mathcal{B}}_{X}^{*})$.
And $H_{1}(\widetilde{\mathcal{B}}_{X}^{*}\times_{G}EG,\mathbb{Z})$, $H^{2}(\widetilde{\mathcal{B}}_{X}^{*}\times_{G}EG,\mathbb{Z})$ are torsion-free.
%\begin{equation}H_{1}(\widetilde{\mathcal{B}}_{X}^{*}\times_{G}EG,\mathbb{Z})\cong \pi_{1}(\widetilde{\mathcal{B}}_{X}^{*}\times_{G}EG)
%\cong\pi_{1}(\widetilde{\mathcal{B}}^{*}_{X})\times\pi_{1}(BG)\cong\pi_{1}(\widetilde{\mathcal{B}}_{X})\cong\pi_{1}(Map_{E}(X,BSU(N)))
%\nonumber \end{equation}

Thus to prove (\ref{ori of amb space}), we only need
\begin{equation}2c_{1}(det(Ind(\slashed{\mathbb{D}}_{End\mathcal{E}_{X}}))\times_{G} EG)-
c_{1}(r^{*}det(Ind(\slashed{\mathbb{D}}_{End\mathcal{E}_{Y}}))\times_{G} EG)=0\in H^{2}(\widetilde{\mathcal{B}}_{X}^{*}\times_{G}EG,\mathbb{Q}).
\nonumber \end{equation}
We are furthermore left to show
\begin{equation}2c_{1}\big(det(Ind(\slashed{\mathbb{D}}_{End\mathcal{E}_{X}}))\times_{G} EG\big)|_{C}-
c_{1}\big(r^{*}det(Ind(\slashed{\mathbb{D}}_{End\mathcal{E}_{Y}}))\times_{G} EG\big)|_{C}=0\in H^{2}(C,\mathbb{Q})\nonumber \end{equation}
for any embedded surface $C\subseteq \widetilde{\mathcal{B}}_{X}^{*}\times_{G}EG$.

%As an element in the $K$-group of $C$, we have
%\begin{equation}\label{equality 2} Ind_{\mathbb{C}}(\mathcal{B}_{X}^{*})|_{C}=r^{*}Ind_{\mathbb{C}}(\mathcal{B}_{Y}^{*})|_{C}+
%Ind_{\mathbb{C}}^{K_{X}}(\mathcal{B}_{X}^{*})|_{C},   \end{equation}
%where $Ind_{\mathbb{C}}^{K_{X}}(\mathcal{B}_{X}^{*})|_{C}$ is the index bundle twisted by connections from $K_{X}$ with
%$Ind_{\mathbb{C}}^{K_{X}}(\mathcal{B}_{X}^{*})|_{D_{A}}=H^{0,odd}_{D_{A}}(X,EndE\otimes K_{X})-H^{0,even}_{D_{A}}(X,EndE\otimes K_{X})$.
Note that the universal bundle $\mathcal{E}_{X}$ is $G$-invariant and extends to a universal bundle $\mathcal{E}_{X}\times_{G}EG$
over $\widetilde{\mathcal{B}}_{X}^{*}\times_{G}EG$.
We denote the universal bundle over $C$ to be $\mathcal{E}\rightarrow X\times C$, $\pi_{X}: X\times C\rightarrow C$, $\pi_{Y}: Y\times C\rightarrow C$ to be projection maps, and $i=(i_{Y},Id): Y\times C\rightarrow X\times C$. The commutative diagram
\begin{equation}\xymatrix{
  Y\times C \ar[rr]^{i} \ar[dr]_{\pi_{Y}}
                &  &    X\times C \ar[dl]^{\pi_{X}}    \\
                & C                 }
\nonumber \end{equation}
implies that $\pi_{X_{!}}\circ i_{!}=\pi_{Y_{!}}$ for Gysin homomorphisms on cohomologies.
By applying the Atiyah-Singer family index theorem \cite{as4}, we have
\begin{eqnarray*}c_{1}\big(det(Ind(\slashed{\mathbb{D}}_{End\mathcal{E}_{X}}))\times_{G} EG\big)|_{C}
&=& c_{1}(det(Ind(\slashed{\mathbb{D}}_{End(\mathcal{E}_{X}\times_{G}EG)}))|_{C}) \\
&=& \pi_{X_{!}}([ch(End \mathcal{E})\cdot Td(X)]^{(2n+1)}) \\
&=& (\sum_{i=1}^{n}ch_{2i}(End \mathcal{E})\cdot Td_{2n-2i+1}(X))/[X],
\nonumber \end{eqnarray*}
%\begin{eqnarray*}\label{equality 3}c_{1}(Ind^{K_{X}}_{\mathbb{C}}(\mathcal{B}_{X}^{*})|_{C})&=&[ch(End \mathcal{E})\cdot p^{*}ch(K_{X})\cdot Td(X)]^{(5)}/[X] \\ &=& [-\frac{1}{24}ch_{2}(End \mathcal{E})\cdot c_{1}(X)c_{2}(X)-\frac{1}{2}ch_{4}(End \mathcal{E})\cdot c_{1}(X)]/[X]
%\\ &=& -c_{1}(Ind_{\mathbb{C}}(\mathcal{B}_{X}^{*})|_{C}).  \nonumber \end{eqnarray*}
\begin{eqnarray*}c_{1}\big(r^{*}det(Ind(\slashed{\mathbb{D}}_{End\mathcal{E}_{Y}}))\times_{G} EG\big)|_{C}
&=& c_{1}(r^{*}det(Ind(\slashed{\mathbb{D}}_{End(\mathcal{E}_{Y}\times_{G} EG)}))|_{C}) \\
&=&\pi_{Y_{!}}([ch(End (i^{*}\mathcal{E}))\cdot Td(Y)]^{(2n)}) \\
&=& \pi_{X_{!}}\circ i_{!}([i^{*}ch(End \mathcal{E})\cdot \frac{Td(X)|_{Y}}{Td(\mathcal{N}_{Y/X})}]^{(2n)}) \\
&=& \pi_{X_{!}}([i_{!}\circ i^{*}(ch(End \mathcal{E})\cdot \frac{Td(X)}{Td(K_{X}^{-1})})]^{(2n+1)}) \\
&=& \pi_{X_{!}}([ch(End \mathcal{E})\cdot \frac{Td(X)}{Td(K_{X}^{-1})}\cdot c_{1}(K_{X}^{-1})]^{(2n+1)}) \\
&=& [ch(End \mathcal{E})\cdot Td(X)\cdot(1-e^{-c_{1}(X)})]^{(2n+1)}/[X].
%&=& 2c_{1}(Ind_{\mathbb{C}}(\mathcal{B}_{X}^{*})|_{C})
\nonumber \end{eqnarray*}
We introduce $\widetilde{Td}(X)=Td(X)\cdot(1-e^{-c_{1}(X)})$. To prove
\begin{equation}2c_{1}\big(det(Ind(\slashed{\mathbb{D}}_{End\mathcal{E}_{X}}))\times_{G} EG\big)|_{C}=
c_{1}\big(r^{*}det(Ind(\slashed{\mathbb{D}}_{End\mathcal{E}_{Y}}))\times_{G} EG\big)|_{C},  \nonumber \end{equation}
we are left to show
\begin{equation}\label{equality 3}\widetilde{Td}_{2i-1}(X)=2 \textrm{ }Td_{2i-1}(X),  \textrm{ } \textrm{for}  \textrm{ } 1\leq i\leq n,  \end{equation}
i.e. $\widetilde{Td}(X)-2 \textrm{ }Td(X)$ consists of even index classes.
Note that the $\hat{A}$-class satisfies
\begin{equation} Td(X)=e^{\frac{c_{1}(X)}{2}}\cdot\hat{A}(X), \nonumber \end{equation}
and
\begin{eqnarray*}\widetilde{Td}(X)-2Td(X)&=&\hat{A}(X)(e^{\frac{c_{1}(X)}{2}}-e^{-\frac{c_{1}(X)}{2}})-2\hat{A}(X)\cdot e^{\frac{c_{1}(X)}{2}} \\
&=& -\hat{A}(X)(e^{\frac{c_{1}(X)}{2}}+e^{-\frac{c_{1}(X)}{2}}),  \nonumber \end{eqnarray*}
which is of even index as both factors in the RHS are so.
\end{proof}

\subsection{Relations with relative orientations for restriction morphisms}
We start with the restriction morphism $r: \mathcal{M}_{X}\rightarrow\mathcal{M}_{Y}$ between two coarse moduli spaces and determinant line bundles
$\mathcal{L}_{\mathcal{M}_{X}}$, $\mathcal{L}_{\mathcal{M}_{Y}}$ over them respectively.
\begin{lemma}\label{rel ori lemma}
There exists a canonical isomorphism
\begin{equation}\alpha:(\mathcal{L}_{\mathcal{M}_{X}})^{\otimes2}\cong r^{*}\mathcal{L}_{\mathcal{M}_{Y}}   \nonumber \end{equation}
between algebraic line bundles.
\end{lemma}
\begin{proof}
We consider a commutative diagram
\begin{equation}\xymatrix{
  Y\times \mathcal{M}_{X} \ar[rr]^{i} \ar[dr]_{\pi_{Y}}
                &  &    X\times \mathcal{M}_{X} \ar[dl]^{\pi_{X}}    \\
                & \mathcal{M}_{X}. }
\nonumber \end{equation}
By definition, $r^{*}\mathcal{L}_{\mathcal{M}_{Y}}=det((R\pi_{Y})_{*}(R\mathcal{H}om(i^{*}\mathcal{E},i^{*}\mathcal{E})))$,
where $\mathcal{E}\rightarrow X\times \mathcal{M}_{X}$ is the universal bundle of $\mathcal{M}_{X}$.
By the adjunction formula (see for instance \cite{huy1}), we have
\begin{eqnarray*}(R\pi_{Y})_{*}(R\mathcal{H}om(i^{*}\mathcal{E},i^{*}\mathcal{E}))&=&(R\pi_{X})_{*}\circ i_{*}(i^{*}R\mathcal{H}om(\mathcal{E},\mathcal{E}))   \\
&=& (R\pi_{X})_{*}(R\mathcal{H}om(\mathcal{E},\mathcal{E})\otimes^{L}\mathcal{O}_{Y\times \mathcal{M}_{X}}).
\nonumber \end{eqnarray*}
From the short exact sequence $0\rightarrow p^{*}K_{X}\rightarrow\mathcal{O}_{X\times \mathcal{M}_{X}}
\rightarrow\mathcal{O}_{Y\times \mathcal{M}_{X}}\rightarrow0$, where $p: X\times \mathcal{M}_{X}\rightarrow X$ is the projection,
we obtain exact triangles
\begin{equation}R\mathcal{H}om(\mathcal{E},\mathcal{E})\otimes p^{*}K_{X}\rightarrow R\mathcal{H}om(\mathcal{E},\mathcal{E})\rightarrow
R\mathcal{H}om(\mathcal{E},\mathcal{E})\otimes^{L}\mathcal{O}_{Y\times \mathcal{M}_{X}}, \nonumber \end{equation}
\begin{equation}\label{equ 3}(R\pi_{X})_{*}(R\mathcal{H}om(\mathcal{E},\mathcal{E})\otimes p^{*}K_{X})\rightarrow (R\pi_{X})_{*}
R\mathcal{H}om(\mathcal{E},\mathcal{E})\rightarrow (R\pi_{X})_{*}(R\mathcal{H}om(\mathcal{E},\mathcal{E})\otimes^{L}\mathcal{O}_{Y\times \mathcal{M}_{X}}).
\end{equation}
By the Grothendieck-Serre duality \cite{conrad}, we have
\begin{equation}\label{equ 4}(R\pi_{X})_{*}(R\mathcal{H}om(\mathcal{E},\mathcal{E})\otimes p^{*}K_{X}[2n])\simeq \big((R\pi_{X})_{*}
R\mathcal{H}om(\mathcal{E},\mathcal{E})\big)^{\vee}.  \end{equation}
By taking the corresponding element of (\ref{equ 3}) in the Grothendieck group (see for instance page 124 of \cite{huy1}),
and using (\ref{equ 4}), we can obtain an isomorphism of determinant line bundles
\begin{equation}(\mathcal{L}_{\mathcal{M}_{X}})^{\otimes2}\cong r^{*}\mathcal{L}_{\mathcal{M}_{Y}}.\nonumber \end{equation}
\end{proof}
\begin{remark} ${}$ \\
1. When $dim_{\mathbb{C}}X=2$ and $Y$ is an elliptic curve, $\mathcal{M}_{X}$ is smooth by Serre duality. The above isomorphism gives
$(K_{\mathcal{M}_{X}})^{\otimes2}\cong \mathcal{O}_{\mathcal{M}_{X}}$. This fits into the work of Bottacin \cite{bottacin} on the existence
of Poisson structures on moduli spaces of simple sheaves on Poisson surfaces. By Proposition 6.2 \cite{bottacin}, the Poisson structure
on $\mathcal{M}_{X}$ is non-degenerate and defines a holomorphic symplectic structure if $r$ is a well-defined map to a coarse moduli of simple bundles
on $Y$.  \\
2. Similar isomorphism holds true for Lagrangians in $(3-2n)$-shifted symplectic derived schemes
in the sense of Pantev, T\"{o}en, Vaqui\'{e} and Vezzosi \cite{ptvv}.
\end{remark}
Then it is natural to make the following definition for orientations in this relative set-up (compare to Definition 2.11 of Borisov and Joyce \cite{bj}).
\begin{definition}\label{def on rel ori}
Let $X$ be a smooth projective $2n$-fold with a smooth anti-canonical divisor $Y\in|K^{-1}_{X}|$, and $r: \mathcal{M}_{X}\rightarrow\mathcal{M}_{Y}$
%\begin{equation}r: \mathcal{M}_{X}\rightarrow\mathcal{M}_{Y}  \nonumber \end{equation}
be a well-defined restriction morphism between coarse moduli spaces of simple sheaves on $X$ and $Y$ with fixed Chern classes respectively.

A \emph{relative orientation} for morphism $r$ consists of a square root $(\mathcal{L}_{\mathcal{M}_{Y}}|_{\mathcal{M}^{red}_{Y}})^{\frac{1}{2}}$
of the determinant line bundle $\mathcal{L}_{\mathcal{M}_{Y}}|_{\mathcal{M}^{red}_{Y}}$ and an isomorphism
\begin{equation}\theta:\mathcal{L}_{\mathcal{M}_{X}}|_{\mathcal{M}^{red}_{X}} \cong
r^{*}(\mathcal{L}_{\mathcal{M}_{Y}}|_{\mathcal{M}^{red}_{Y}})^{\frac{1}{2}}  \nonumber \end{equation}
such that $\theta\otimes\theta\cong \alpha$ holds over $\mathcal{M}^{red}_{X}$ for the isomorphism $\alpha$ in Lemma \ref{rel ori lemma}.
\end{definition}
\begin{remark} ${}$ \\
1. When $X$ is a $CY_{2n}$, $Y=\emptyset$, the above definition coincides with Definition 2.11 of \cite{bj} on orientations of even-shifted symplectic
derived schemes, because a real line bundle is trivial on a scheme if and only if it is trivial over its reduced scheme
(see also Theorem \ref{ori on even CY}). \\
2. When $dim_{\mathbb{C}}X=4$, relative orientations are required to defined relative $DT_{4}$ invariants \cite{caoleung2}.
\end{remark}
The previous Theorem \ref{ori on rel} gives an evidence for the existence of relative orientations. Another partial result is given as follows.
\begin{proposition}\label{prop on rel ori}
We assume $H^{1}(\mathcal{M}_{X},\mathbb{Z}_{2})=0$. Then relative orientations
for restriction morphism $r: \mathcal{M}_{X}\rightarrow\mathcal{M}_{Y}$ exist.
\end{proposition}
\begin{proof}
We denote the uniquely determined restriction morphism between reduced schemes also by $r:\mathcal{M}^{red}_{X}\rightarrow\mathcal{M}^{red}_{Y}$.
Combining with Lemma \ref{rel ori lemma}, there exists a canonical isomorphism
\begin{equation}\alpha:(\mathcal{L}_{\mathcal{M}_{X}}|_{\mathcal{M}^{red}_{X}})^{\otimes2}\cong
r^{*}(\mathcal{L}_{\mathcal{M}_{Y}}|_{\mathcal{M}^{red}_{Y}}).   \nonumber \end{equation}
We take any square root $(\mathcal{L}_{\mathcal{M}_{Y}}|_{\mathcal{M}^{red}_{Y}})^{\frac{1}{2}}$
of $\mathcal{L}_{\mathcal{M}_{Y}}|_{\mathcal{M}^{red}_{Y}}$ (existence is due to \cite{no}).
Thus $\alpha$ determines a trivialization
\begin{equation}\alpha:\big((\mathcal{L}_{\mathcal{M}_{X}}|_{\mathcal{M}^{red}_{X}})\otimes
r^{*}(\mathcal{L}_{\mathcal{M}_{Y}}|_{\mathcal{M}^{red}_{Y}})^{-\frac{1}{2}}\big)^{\otimes2}\cong\mathcal{O}_{\mathcal{M}_{X}^{red}}. \nonumber \end{equation}
From the short exact sequence
\begin{equation} \xymatrix@1{1\rightarrow\mathbb{Z}_{2}\rightarrow\mathcal{O}^{*}_{\mathcal{M}_{X}^{red}}\ar[r]^{\quad f\mapsto f^{2}} &
\mathcal{O}^{*}_{\mathcal{M}_{X}^{red}}\rightarrow 1 }, \nonumber \end{equation}
we obtain an exact sequence
\begin{equation}\xymatrix@1{0\rightarrow H^{0}(\mathcal{M}_{X}^{red},\mathbb{Z}_{2})\rightarrow H^{0}(\mathcal{M}_{X}^{red},\mathcal{O}^{*}_{\mathcal{M}_{X}^{red}})
\ar[r]^{\quad \quad\quad \quad f\mapsto f^{2}} & H^{0}(\mathcal{M}_{X}^{red},\mathcal{O}^{*}_{\mathcal{M}_{X}^{red}})\rightarrow  } \nonumber \end{equation}
\begin{equation}\xymatrix@1{\rightarrow H^{1}(\mathcal{M}_{X}^{red},\mathbb{Z}_{2})\ar[r]^{i} &
H^{1}(\mathcal{M}_{X}^{red},\mathcal{O}^{*}_{\mathcal{M}_{X}^{red}}) \ar[r]^{L\mapsto L^{2}\quad } &
H^{1}(\mathcal{M}_{X}^{red},\mathcal{O}^{*}_{\mathcal{M}_{X}^{red}})\rightarrow\cdot\cdot\cdot }.   \nonumber \end{equation}
$H^{1}(\mathcal{M}_{X},\mathbb{Z}_{2})=0$ implies that any square root of $\mathcal{O}_{\mathcal{M}_{X}^{red}}$ is $\mathcal{O}_{\mathcal{M}_{X}^{red}}$.
\end{proof}

\section{Appendix}

\subsection{Some basic facts in spin geometry}
In this section, we recall some basic facts in spin geometry.
The main references are the book \cite{lawson} of Lawson and Michelson and the book \cite{adams} by Adams.
%We first come to the fundamental representation of the spin group $Spin(n)$.
\begin{theorem}\label{rep of spin gp}(Theorem 4.6 \cite{adams}) ${}$ \\
Let $\Delta_{\mathbb{C}}^{\pm}$ be two fundamental complex spinor representations of $Spin(n)$, where $n$ is even. \\
Then, we have \\
(1) $\Delta_{\mathbb{C}}^{\pm}$ are real if $n=8k$; \\
(2) $\Delta_{\mathbb{C}}^{\pm}$ are symplectic if $n=8k+4$; \\
(3) $\Delta_{\mathbb{C}}^{+}\cong (\Delta_{\mathbb{C}}^{-})^{*}$ if $n=8k+2$ or $8k+6$.
\end{theorem}
\begin{remark}
In case (1), $\Delta_{\mathbb{C}}^{\pm}$ are endowed with $Spin(n)$-invariant non-degenerate quadratic forms
as they are complexfications of real representations. \\
In case (2), $\Delta_{\mathbb{C}}^{\pm}$ are endowed with $Spin(n)$-invariant non-degenerate $2$-forms.
\end{remark}
Let $(X,g,P_{Spin}(TX))$ be an even dimensional spin manifold. We define its complex spinor bundles by
\begin{equation}\slashed{S}^{\pm}_{\mathbb{C}}(X)=P_{Spin}(TX)\times_{Spin(n)}\Delta_{\mathbb{C}}^{\pm}.   \nonumber \end{equation}
As $\slashed{S}^{\pm}_{\mathbb{C}}(X)$ are Clifford bundles, and the Levi-Civita connection on $(X,g)$ induces connections $\nabla$'s on
complex vector bundles $\slashed{S}^{\pm}_{\mathbb{C}}(X)$, there exists Dirac operators
\begin{equation}\slashed{D}^{\pm}: \Gamma(\slashed{S}^{\pm}_{\mathbb{C}}(X))\rightarrow \Gamma(\slashed{S}^{\mp}_{\mathbb{C}}(X)). \nonumber \end{equation}
If $dim_{\mathbb{R}}(X)=8k$, by Theorem \ref{rep of spin gp}, $\slashed{S}^{\pm}_{\mathbb{C}}(X)$ are the complexificaiton of
real spinor bundles $\slashed{S}^{\pm}(X)$. The corresponding Dirac operators (with their kernel and cokernel) are the complexifications of the real ones.

If $dim_{\mathbb{R}}(X)=8k+2$ or $8k+6$, by Theorem \ref{rep of spin gp}, two complex spinor bundles  $\slashed{S}^{\pm}_{\mathbb{C}}(X)$ are dual to each other as bundles with left $Cl(X)$-action, i.e. there is an isomorphism
\begin{equation}\label{equ 2}\slashed{S}^{+}_{\mathbb{C}}(X)\cong (\slashed{S}^{-}_{\mathbb{C}}(X))^{*}, \end{equation}
which is equivariant under the action of Clifford bundle $Cl(X)$.

Let $(E,h)\rightarrow X$ be a Hermitian complex vector bundle, and $A$ be an unitary connection.
We can define twisted Dirac operators
\begin{equation}\slashed{D}^{+}_{A}: \Gamma(\slashed{S}^{+}_{\mathbb{C}}(X)\otimes E)\rightarrow \Gamma(\slashed{S}^{-}_{\mathbb{C}}(X)\otimes E),
\nonumber \end{equation}
\begin{equation}\slashed{D}^{-}_{A^{*}}: \Gamma(\slashed{S}^{-}_{\mathbb{C}}(X)\otimes E^{*})\rightarrow
\Gamma(\slashed{S}^{+}_{\mathbb{C}}(X)\otimes E^{*}). \nonumber \end{equation}
As $\slashed{S}^{+}_{\mathbb{C}}(X)$ and $\slashed{S}^{+}_{\mathbb{C}}(X)$ are dual as Clifford bundles, when they are coupled with
dual bundles $E$ and $E^{*}$, kernels of the corresponding twisted Dirac operators are dual to each other, i.e.
\begin{theorem}\label{thm 8k+2}
Let $X$ be a compact spin manifold with $dim_{\mathbb{R}}(X)=8k+2$ or $8k+6$, and $E\rightarrow X$ be a
(Hermitian) complex vector bundle with an unitary connection $A$. Then there is a canonical isomorphism
\begin{equation}ker(\slashed{D}^{+}_{A})\cong(ker(\slashed{D}^{-}_{A^{*}}))^{*}  \nonumber \end{equation}
of complex vector spaces.
\end{theorem}
As a direct corollary of Theorem \ref{thm 8k+2}, we have
\begin{corollary}\label{cor}
Let $X$ be a compact spin manifold with $dim_{\mathbb{R}}(X)=8k+2$ or $8k+6$, and $(E,h)$ be a (Hermitian) complex vector bundle
with an unitary connection $A$. Then for Dirac operator
\begin{equation}\slashed{D}^{+}_{A\oplus A^{*}}: \Gamma(\slashed{S}^{+}_{\mathbb{C}}(X)\otimes T^{*}E)
\rightarrow\Gamma(\slashed{S}^{-}_{\mathbb{C}}(X)\otimes T^{*}E), \nonumber \end{equation}
we have a canonical isomorphism
\begin{equation}
det(Ind(\slashed{D}^{+}_{A\oplus A^{*}}))\cong (det(Ind(\slashed{D}^{+}_{A})))^{\otimes2}
\nonumber \end{equation}
for determinant lines.
\end{corollary}
\begin{remark}
The above canonical isomorphisms naturally extend to corresponding determinant line bundles.
\end{remark}

\subsection{Some standard material from gauge theory}
In this subsection, we recall some standard material from gauge theory. The main references are the book \cite{dk} by Donaldson and Kronheimer
and a series of papers \cite{d}, \cite{d0}, \cite{d1} by Donaldson.

We fix a compact spin manifold $X$ of even dimension and a Hermitian complex vector bundle $(E,h)\rightarrow X$.
Given an unitary connection $A$ on $E$, we can define the twisted Dirac operator
\begin{equation}\slashed{D}_{A^{*}\otimes A}: \Gamma(\slashed{S}^{+}_{\mathbb{C}}(X)\otimes EndE)\rightarrow
\Gamma(\slashed{S}^{-}_{\mathbb{C}}(X)\otimes EndE) \nonumber \end{equation}
following Theorem 13.10 of \cite{lawson}. $[ker(\slashed{D}_{A^{*}\otimes A})-coker(\slashed{D}_{A^{*}\otimes A})]$ exists as
an element in the $K$-theory $K(pt)$ of one point, and there is a family version of the above construction as follows.

Let $\mathcal{A}$ be the space of all unitary connections on $(E,h)$, and $\mathcal{G}$ be the group of unitary gauge transformations.
If we denote the $U(r)$-principal bundle of $E$ by $P$,
\begin{equation}\mathcal{G}=\Gamma(X,P\times_{U(r)} U(r)) \nonumber \end{equation}
is the space of $C^{\infty}$-sections of the bundle $P\times_{U(r)} U(r)$ (with the conjugate action $U(r)\curvearrowright U(r)$).
The orbit space $\mathcal{B}=\mathcal{A}/\mathcal{G}$ (with suitable Sobolev structure) exists as a metrizable topological space (Lemma 4.2.4 \cite{dk}). Following \cite{dk}, we fix a base point $x_{0}\in X$ and introduce the space
\begin{equation}\widetilde{\mathcal{B}}_{X}=\mathcal{A}\times_{\mathcal{G}}P_{x_{0}}   \nonumber \end{equation}
of equivalent classes of framed connections. Equivalently, $\widetilde{\mathcal{B}}_{X}=\mathcal{A}/\mathcal{G}_{0}$,
where $\mathcal{G}_{0}\vartriangleleft \mathcal{G}$ is the subgroup of gauge transformations which fix the fiber $P_{x_{0}}$.
As $\mathcal{G}_{0}$ acts on $\mathcal{A}$ freely, $\widetilde{\mathcal{B}}_{X}$ (with suitable Sobolev structure)
has a Banach manifold structure whose weak homotopy type will not depend on the chosen Sobolev structures (Proposition 5.1.4 \cite{dk}).
Meanwhile, there exists a universal bundle $\mathcal{E}=\mathcal{A}\times_{\mathcal{G}_{0}}E$ over $\widetilde{\mathcal{B}}_{X}\times X$
(trivialized on $\widetilde{\mathcal{B}}_{X}\times\{x_{0}\}$), which carries a universal family of framed connections.

We then couple the Dirac operator $\slashed{D}$ on $X$ with the universal connection on $\mathcal{E}$ and there is an index bundle
\begin{equation} Ind(\slashed{\mathbb{D}}_{End\mathcal{E}})\in K(\widetilde{\mathcal{B}}_{X}), \nonumber \end{equation}
which satisfies $Ind(\slashed{\mathbb{D}}_{End\mathcal{E}})|_{[A]}=ker(\slashed{D}_{A^{*}\otimes A})-coker(\slashed{D}_{A^{*}\otimes A})$ \cite{as4}
(see also page 181 of \cite{dk}).
%$\mathcal{B}_{X}^{*}$ is an open subset of the orbit space $\mathcal{B}_{X}=\mathcal{A}/\mathcal{G}^{0}$ whose complement is of infinite codimension
For any finite dimensional submanifold $C\subseteq \widetilde{\mathcal{B}}_{X}$ and the induced family $\mathbb{E}=\mathcal{E}|_{C}$ over
$C\times X$, we have the Atiyah-Singer family index formula \cite{as4}:
\begin{equation}ch(Ind(\slashed{\mathbb{D}}_{End\mathbb{E}}))=(ch(End\mathbb{E})\cdot \hat{A}(X))/[X].   \nonumber \end{equation}

\subsection{Seidel-Thomas twists}
In this section, we recall the Seidel-Thomas twist \cite{st} and how it could be used to identify a moduli space of simple sheaves to
a moduli space of simple holomorphic bundles, which is the work of Joyce and Song \cite{js}.
\begin{definition}
Let $(X,\mathcal{O}_{X}(1))$ be a projective Calabi-Yau $m$-fold with $Hol(X)=SU(m)$. For each $n\in \mathbb{Z}$,
the Seidel-Thomas twist $T_{\mathcal{O}_{X}(-n)}$ by $\mathcal{O}_{X}(-n)$ is the Fourier-Mukai transform from $D(X)$ to $D(X)$ with kernel
\begin{equation}K=cone(\mathcal{O}_{X}(n)\boxtimes\mathcal{O}_{X}(-n)\rightarrow \mathcal{O}_{\Delta}).
\nonumber \end{equation}
\end{definition}
In general, $T_{n}\triangleq T_{\mathcal{O}_{X}(-n)}[-1]$ maps sheaves to complexes of sheaves. But for $n\gg 0$, we have
\begin{theorem}\label{seidel thomas twist lemma}(Joyce-Song, Lemma 8.2 of \cite{js})
Let $U$ be a finite type $\mathbb{C}$-scheme and $\mathcal{F}_{U}$ is a coherent sheaf on $U\times X$ flat over $U$ i.e.
it is a $U$-family of coherent sheaves on $X$. Then for $n\gg 0$,
$T_{n}(\mathcal{F}_{U})$ is also a $U$-family of coherent sheaves on $X$.
\end{theorem}
Sufficiently many compositions of Seidel-Thomas twists map sheaves to vector bundles.
\begin{definition}
For a nonzero coherent sheaf $\mathcal{F}$, the homological dimension $hd(\mathcal{F})$ is the smallest $n\geq0$ for
which there exists an exact sequence in the abelian category $coh(X)$ of coherent sheaves
\begin{equation}0\rightarrow E_{n}\rightarrow E_{n-1}\cdot\cdot\cdot\rightarrow E_{0}\rightarrow \mathcal{F}\rightarrow 0
\nonumber \end{equation}
with $\{E_{i}\}_{i=0,...,n}$ are vector bundles.
\end{definition}
\begin{theorem}(Joyce-Song, Lemma 8.4 of \cite{js})
Let $\mathcal{F}_{U}$, $n\gg0$ be the same as in Theorem \ref{seidel thomas twist lemma}, then for any $u\in U$,
we have $hd(T_{n}(\mathcal{F}_{u}))=max(hd(\mathcal{F}_{u})-1,0)$.
\end{theorem}
\begin{corollary}(Joyce-Song, Corollary 8.5 of \cite{js})
Let $U$ be a finite type $\mathbb{C}$-scheme and $\mathcal{F}_{U}$ is a $U$-family of coherent sheaves on $X$.
Then there exists $n_{1},...n_{m}\gg0$ such that for $T_{n_{m}}\circ\cdot\cdot\cdot\circ T_{n_{1}}(\mathcal{F}_{U})$
is a $U$-family of vector bundles on $X$.
\end{corollary}
Meanwhile, Seidel-Thomas twists are auto-equivalences of derived category $D(X)$, they preserve determinant line bundles (if exists)
of corresponding moduli spaces.
\begin{corollary}(Joyce-Song \cite{js})\label{st preserves L}
Given a coarse moduli space $\mathcal{M}_{X}$ of simple sheaves with fixed Chern classes, we choose sufficiently large integers
$n_{1},...n_{m}\gg0$ such that $\Psi\triangleq T_{n_{m}}\circ\cdot\cdot\cdot\circ T_{n_{1}}$ identifies $\mathcal{M}_{X}$
with a coarse moduli space $\mathfrak{M}^{bdl}_{X}$ of simple holomorphic bundles. Then
\begin{equation}\Psi^{*}\mathcal{L}_{\mathfrak{M}^{bdl}_{X}}\cong\mathcal{L}_{\mathcal{M}_{X}}, \nonumber \end{equation}
where $\mathcal{L}_{\bullet}$ is the determinant line bundle of the corresponding moduli space.

Moreover, if $X$ is a $CY_{2n}$, $\mathcal{L}_{\mathfrak{M}^{bdl}_{X}}$ and $\mathcal{L}_{\mathcal{M}_{X}}$ are endowed with
non-degenerate quadratic forms from Serre duality pairing. The isomorphism $\Psi^{*}\mathcal{L}_{\mathfrak{M}^{bdl}_{X}}
\cong\mathcal{L}_{\mathcal{M}_{X}}$ also preserves
the quadratic forms.
\end{corollary}

%Address: The Institute of Mathematical Sciences and Department of Mathematics, The Chinese University of Hong Kong, Shatin, Hong Kong \\
%Email: ylcao@math.cuhk.edu.hk ; leung@math.cuhk.edu.hk

\end{document}